\documentclass[reqno]{amsart}

\usepackage{amsmath,amssymb,amscd,amsthm,accents} \usepackage{graphicx}
\DeclareGraphicsExtensions{.eps} \usepackage{mathrsfs}
\usepackage[mathcal]{eucal}
\usepackage{enumitem}

\def\sideremark#1{\ifvmode\leavevmode\fi\vadjust{\vbox to0pt{\vss 
			\hbox to 0pt{\hskip\hsize\hskip1em          
				\vbox{\hsize3cm\tiny\raggedright\pretolerance10000%             
					\noindent #1\hfill}\hss}\vbox to8pt{\vfil}\vss}}}%

\newtheorem{Thm}{Theorem}{\bfseries}{\itshape}
\newtheorem*{Thm*}{Theorem}{\bfseries}{\itshape}
\newtheorem{Cor}{Corollary}{\bfseries}{\itshape}
{\bfseries}{\itshape}
\newtheorem{Lem}[Cor]{Lemma}{\bfseries}{\itshape}
\newtheorem*{Lem*}{Lemma}{\bfseries}{\itshape}
{\bfseries}{\itshape}
{\bfseries}{\itshape}
\newtheorem{Def}[Cor]{Definition}{\bfseries}{\rmfamily}
{\scshape}{\rmfamily}
\newtheorem{Rem}[Cor]{Remark}{\scshape}{\rmfamily}
{\bfseries}{\itshape}

{\scshape}{\rmfamily}

\renewcommand\ge{\geqslant} \renewcommand\le{\leqslant}
\let\tildeaccent=\~ \let\hataccent=\^
\renewcommand\~[1]{\widetilde{#1}}

\def\<{\left<} \def\>{\right>} \def\({\left(} \def\){\right)}

\def\norm#1{\Vert #1\Vert}

\let\parasymbol=\S \def\secref#1{\parasymbol\ref{#1}}

\let\polishL=l \def\Zoladek.{\.Zol\c adek}

 \def\const{\operatorname{const}}

\def\Re{\operatorname{Re}} 
 \def\dist{\operatorname{dist}}

\def\diam{\operatorname{diam}} 
 
\def\etc.{\emph{etc}.}

 \def\R{{\mathbb R}} \def\C{{\mathbb C}}  \def\N{{\mathbb N}}  
\def\H{{\mathbb H}}

 \def\e{\varepsilon} \def\S{\varSigma}

 \def\d{\,\mathrm d}

 \def\cL{{\mathcal L}} 
  
  \def\cC{{\mathcal C}}

\def\cM{{\mathcal M}}

\def\id{\operatorname{id}}

\def\va{{\mathbf a}}%

\def\vx{{\mathbf x}}

\def\vp{{\mathbf p}}

\def\valpha{{\boldsymbol\alpha}}
\def\vbeta{{\boldsymbol\beta}}

\def\vphi{{\boldsymbol\phi}}

\def\an{\omega}

\DeclareMathOperator{\gr}{gr}
\DeclareMathOperator{\Th}{Th}

\def\rmF{{\mathrm{F}}}
\def\rmS{{\mathrm{S}}}

\begin{document}

% +Title
\title{The Yomdin-Gromov algebraic lemma revisited}

\author{Gal Binyamini}
\address{Weizmann Institute of Science, Rehovot, Israel}
\email{gal.binyamini@weizmann.ac.il}

\author{Dmitry Novikov}
\address{Weizmann Institute of Science, Rehovot, Israel}
\email{dmitry.novikov@weizmann.ac.il}
\thanks{This research was supported by the ISRAEL SCIENCE FOUNDATION
	(grant No. 1167/17) and by funding received from the MINERVA
	Stiftung with the funds from the BMBF of the Federal Republic of
	Germany. This project has received funding from the European Research Council (ERC) under the European Union's Horizon 2020 research and innovation programme (grant agreement No 802107)}

\subjclass[2010]{	03C64, 	37B40, 14P10, 11U09 }
\keywords{Gromov-Yomdin parametrization, o-minimality}
\begin{abstract}
  In 1987, Yomdin proved a lemma on smooth parametrizations of
  semialgebraic sets as part of his solution of Shub's entropy
  conjecture for $C^\infty$ maps. The statement was further refined by
  Gromov, producing what is now known as the Yomdin-Gromov algebraic
  lemma. Several complete proofs based on Gromov's sketch have appeared
  in the literature, but these have been considerably more complicated
  than Gromov's original presentation due to some technical issues.

  In this note we give a proof that closely follows Gromov's original
  presentation. We prove a somewhat stronger statement, where the
  parameterizing maps are guaranteed to be \emph{cellular}. It turns
  out that this additional restriction, along with some elementary
  lemmas on differentiable functions in o-minimal structures, allows the
  induction to be carried out without technical difficulties.
\end{abstract}
\maketitle
 
\section{The Yomdin-Gromov lemma}

For a $C^r$-smooth function $f:U\to\R^n$ on a domain $U\subset\R^m$ we
denote by $\norm{f}$ the maximum norm on $U$ and
\begin{equation}
  \norm{f}_r := \max_{|\valpha|\le r} \frac{\norm{D^\valpha f}}{\valpha!}.
\end{equation}

In his work on Shub's entropy conjecture, Yomdin
\cite{yomdin:entropy,yomdin:gy} proved a lemma on $C^r$-smooth
parametrizations of semialgebraic sets. This was further refined by
Gromov in \cite{gromov:gy}, see also \cite{burguet:alg-lemma}, with
the following formulation now known as the Yomdin-Gromov algebraic
lemma.

\begin{Thm}[Yomdin-Gromov algebraic lemma]\label{thm:yg}
  Let $X\subset[0,1]^n$ be a semialgebraic set of dimension $\mu$
  defined by conditions $p_j(\vx)=0$ or $p_j(\vx)<0$ where $p_j$ are
  polynomials and $\sum\deg p_j=\beta$. Let $r\in\N$. There exists a
  constant $C=C(n,\mu,r,\beta)$ and semialgebraic maps
  $\phi_1,\ldots,\phi_C:(0,1)^\mu\to X$ such that their images cover
  $X$ and $\norm{\phi_j}_r\le 1$ for $j=1,\ldots,C$.
\end{Thm}

Pila and Wilkie later realized that this theorem has remarkable
applications in the seemingly unrelated area of Diophantine
approximation. For the generality required by these applications, they
stated and proved an analog of the algebraic lemma for general
o-minimal structures {\cite[Theorem~2.3]{pila-wilkie}} (see
\cite{vdd:book} for general background on o-minimal geometry).

\begin{Thm*}[Pila-Wilkie's version of Yomdin-Gromov]\label{thm:yg-pw}
  Let $X=\{X_p\subset[0,1]^n\}$ be a family of sets definable in an
  o-minimal structure, with $\dim X_p\le\mu$. There exists a constant
  $C=C(X,r)$ such that for any $p$ there exist definable maps
  $\phi_1,\ldots,\phi_C:(0,1)^\mu\to X_p$ such that their images cover
  $X_p$ and $\norm{\phi_j}_r\le 1$ for $j=1,\ldots,C$.
\end{Thm*}

In addition to Pila-Wilkie's proof, Burguet \cite{burguet:alg-lemma}
has also given a proof in the semialgebraic setting around the same
time. Both of these proofs roughly follow Gromov's presentation, but
the technical details are significantly more involved. This is due to
an issue with potentially unbounded derivatives that was not
explicitly treated in Gromov's text, see the first paragraph of
\cite[Section 4]{pila-wilkie}. In both Pila-Wilkie's and Burguet's
papers, the problem is resolved by an additional approximation
argument on $C^r$-smooth maps. We also remark that Kocel-Cynk,
Pawlucki and Vallete have given a proof based on a somewhat different
approach in the general o-minimal setting \cite{kpv:alg-lemma}.

In this paper we give a formal treatment of Gromov's original
proof. In particular, we introduce a slightly stronger notion of
\emph{cellular parametrizations} in Definition~\ref{def:cell-param},
and prove the algebraic lemma with the additional requirement that the
parameterizing maps are cellular. This, in combination with some
elementary lemmas on differentiable functions in o-minimal structures
(see~\secref{sec:bounded-derivatives}), allows us to recover Gromov's
original inductive argument without any technical complications.

\begin{Rem}[On the asymptotic constants]
  The constants $C(X,r)$ and $C(n,\mu,r,\beta)$ in these statements are
  purely existential, and one could ask about their dependence on $r$
  and on the complexity $\beta$ in semialgebraic case or, more
  generally, whenever this complexity can be defined (e.g. Pfaffian
  sets). A good understanding of these constants plays a crucial role
  in some potential applications of the algebraic lemma, both in
  dynamics and in Diophantine approximation. We refer the reader to
  \cite[Section 1]{CCStructures} for a discussion of these
  applications.

  We briefly summarize the current state of the art. Gromov's
  presentation \cite{gromov:gy} gives the polynomial dependence on
  $\beta$ in the semialgebraic case (but no explicit dependence on
  $r$). Cluckers, Pila and Wilkie \cite{cpw:params} give polynomial
  dependence on $r$ for globally subanalytic (and slightly more
  general) sets, with no explicit dependence on $\beta$. In
  \cite{CCStructures} we give a result with polynomial dependence on
  both $r$ and $\beta$ in the semialgebraic case: this is the
  statement which is most useful in the potential applications. We
  also give polynomial dependence on $r$ in the globally subanalytic
  case. In a work in progress of the first author and Jones, Schmidt
  and Thomas, a bound polynomial in $\beta$ (but not in $r$) is
  established for sets definable using restricted-Pfaffian
  functions. This is based on a suitable adaptation of the approach
  presented in the present paper to the restricted Pfaffian structure.
\end{Rem}

\subsection{Statement of the main result}\label{ssec:main result}

We prove a refined version of the Yomdin-Gromov algebraic lemma for
general o-minimal structures using the notion of cellular
parametrizations introduced below. To simplify the terminology for
readers not familiar with o-minimal structures, we will assume
everywhere below that we are working with an o-minimal structure over
the reals $\R$. However all the proofs carry over to the general case
without change.

We denote $I:=(0,1)$. For a vector $\vx_{1..\ell}\in\R^\ell$ we denote
by $\vx_{1..i}$ the vector consisting of its first $i$ coordinates.

\begin{Def}
  A \emph{basic cell} $\cC\subset\R^\ell$ of \emph{length} $\ell$ is a product of $\ell$
  finite intervals $I$ and singletons $\{0\}$. A continuous map
  $f:\cC\to\R^\ell$ is called cellular if for every $i=1,\dots,\ell$ 
  \begin{itemize}
  \item  $f_i(\vx_{1..\ell})=f_i(\vx_{1..i})$, i.e. $f_i$ depends only on the first $i$ coordinates of $\vx$, and 
  \item $f_i(\vx_{1..i-1},\cdot)$ is strictly increasing for every
    $\vx_{1..i-1}\in\cC_{1..i-1}$ (where the cell $\cC_{1..i-1}$ is
    the coordinate projection of $\cC$ to
    $\R^{i-1}=\{x_i=\dots=x_\ell=0\}\subset\R^\ell$).
  \end{itemize} 
\end{Def}

Note in particular that cellular maps preserve dimension and the
composition of cellular maps is cellular.

\begin{Def}\label{def:cell-param}
  A cellular $r$-parametrization of a definable set $X\subset\R^\ell$
  is a collection $\Phi=\{\phi_\alpha:\cC_\alpha\to X\}$ of definable
  cellular $C^r$-smooth maps $\phi_\alpha$ with
  $\norm{\phi_\alpha}_r\le 1$ such that
  $X=\cup_\alpha \phi_\alpha(\cC_\alpha)$.

  A cellular $r$-parametrization of a definable map $F:X\to Y$ is a
  cellular $r$-parametrization $\Phi$ of $X$ satisfying
  $\norm{\phi_\alpha^*F}_r\le 1$ for every
  $\phi_\alpha\in\Phi$.
\end{Def}

\begin{Rem}\label{rem:param-set-vs-funcs}
  Let $X\subset\R^\ell$, $Y\subset\R^q$ and $F:X\to Y$ a definable
  map.  Then $\{\phi_\alpha:\cC_\alpha\to X\}$ is a cellular
  $r$-parametrization of a $F$ if and only if
  $\{(\phi_\alpha,F\circ\phi_\alpha):\cC_\alpha\times\{0\}^n\to \gr F\}$
  is a cellular $r$-parametrization of the graph $\gr F$.
\end{Rem}

We will prove the Yomdin-Gromov lemma in the following form.

\begin{Thm}\label{thm:ygbn}
  Let $\ell,r\in\N$. Then
  \begin{description}
  \item[$\rmS_\ell$] Every definable set $X\subset I^\ell$ admits a
    cellular $r$-parametrization.
  \item[$\rmF_\ell$] Every definable function $F:X\to Y$ with
    $X\subset I^\ell$ and $Y\subset I^q$ (for any $q$) admits a
    cellular $r$-parametrization.
  \end{description}
\end{Thm}

We remark that the cellular formulation of the Yomdin-Gromov lemma
makes it automatically uniform over parameters: a cellular
parametrization of a family with the parameters placed as the initial
variables gives a cellular parametrization of each fiber by
restriction. This uniformity is essential in the applications.

\begin{Rem}
  This exposition appeared first as a part of the course ``Tame
  geometry and applications'' given by authors at the Weizmann
  Institute of Science, Fall 2018.
\end{Rem}

\section{Why $C^r$-smooth?}

Before going into the proof of the Yomdin-Gromov lemma we will address
a natural question. Semialgebraic sets are analytic objects. Why
would one, starting with such tame objects, venture into the far less
rigid smooth category? It would certainly seem natural to expect a far
more rigid parametrization, say by holomorphic maps with respect to
some suitable norm. It turns out that there are deep obstructions
hiding in the background.

Ideally, one would like to replace finite smoothness order $r\in\N$ by
a bound for all derivatives,
\begin{equation}
  \norm{f}_\infty := \sup_{\valpha} \frac{\norm{D^\valpha f}}{\valpha!}.
\end{equation}
If $U\subset\R^m$ and $f:U\to\R^n$ has $\norm{f}_\infty<\infty$ then
$f$ continues holomorphically to a $1$-neighborhood 
$N_1(U)\subset\C^m$ of $U$. Moreover in $N_{1/2}U$ we have
\begin{equation}
  \max_{N_{1/2}(U)} |f| \le \const \norm{f}_\infty.
\end{equation}
So, instead of this $\infty$-norm we might as well use the norm given
by the maximum of the analytic continuation of $f$ to a neighborhood
of some fixed radius. Below we'll write
\begin{equation}
  \norm{f}_\an := \max_{N_1(U)} |f|.
\end{equation}
One would ideally like to prove the Yomdin-Gromov lemma with the maps
$\phi_i$ extendable to $1$-neighborhood of $(0,1)^k$ and with this
stronger norm. Unfortunately this is impossible already for the simple
family of semialgebraic sets (originally considered in this context by
Yomdin),
\begin{equation}
  X_\e = [(-1,1)\times(-1,1)] \cap \{xy=\e\}.
\end{equation}
We'll show that an $\an$-parametrization of $X_\e$ will require
$\log|\log\e|$ maps (so cannot be uniform in the complexity). To
explain this we take a brief detour to the geometry of hyperbolic
Riemann surfaces.

\subsection{Hyperbolic geometry}

Recall that the upper half-plane $\H$ admits a unique hyperbolic
metric of constant curvature $-4$ given by $|\d z|/2y$. A Riemann
surface $U$ is called \emph{hyperbolic} if its universal cover is the
upper half-plane $\H$. In this case $U$ inherits from $\H$ a unique
metric of constant curvature $-4$ which we denote by
$\dist(\cdot,\cdot;U)$ (we sometimes omit $U$ from this notation if it
is clear from the context). By the uniformization theorem, a domain
$U\subset\C$ is hyperbolic if and only if its complement contains at
least two points.

The following is a straightforward consequence of the classical
Schwarz lemma obtained by lifting the map to universal covers.

\begin{Lem}[\protect{Schwarz-Pick \cite[Theorem~2.11]{milnor:dynamics}}]\label{lem:schwarz-pick}
  If $f:S\to S'$ is a holomorphic map between hyperbolic domains
  $S,S'$ then
  \begin{equation}
    \dist(f(p),f(q);S') \le \dist(p,q;S) \qquad \forall p,q\in S.
  \end{equation}
\end{Lem}

\subsection{The obstruction}

Suppose $f:(0,1)\to X_\e$ is a map with $\norm{f}_\an\le 2$. Then $f$
extends analytically to the $1$-neighborhood of $(0,1)$ in $\C$ and is
bounded by $2$ there in absolute value. By analytic continuation, $f$
continues to satisfy $xy=\e$ in $N_1(0,1)$, so
\begin{equation}
  f: N_1(0,1) \to \{xy=\e\} \cap \{|x|,|y|<2\}
\end{equation}
Consider the projection $\pi(x,y)=x$. Then the composition gives a map
\begin{equation}
  \pi\circ f : N_1(0,1) \to \{ \e/2 <|x|<2 \}=A(\e/2,2).
\end{equation}
The domain and the range are hyperbolic domains. So by Schwarz-Pick Lemma~\ref{lem:schwarz-pick} we
have
\begin{equation}
  \diam([\pi\circ f](0,1);A(\e/2,2)) \le \diam((0,1);N_1(0,1)) = \const.
\end{equation}
We see that the set of $x$-s covered by $f$ has bounded hyperbolic
diameter in $A(\e/2,2)$. We would eventually like to cover every
$x\in\pi(X_e)=(\e,1)$. A simple computation gives
\begin{equation}
  \diam((\e,1);A(\e/2,2)) \sim \log|\log\e|
\end{equation}
so indeed at least $\log|\log\e|$ maps will be needed to cover $X_\e$.

\begin{Rem}
  One can show that the bound above is asymptotically sharp,
  i.e. $X_\e$ can indeed be covered by $O(\log|\log\e|)$ maps of unit
  $\an$-norm. Indeed, it suffices to find such a collection of such
  maps from $(0,1)$ to $X_\e$ which extend analytically to the complex
  disc $D(2)$, with both coordinates bounded by $2$ on this
  disc. Equivalently by considering only the $x$-coordinate, we may
  look for a collection of maps from $(0,1)$ into $(\e,1)$ which
  extend to maps $D(2)\to A(\e/2,2)$. Passing to the logarithmic
  chart, we seek maps from $D(2)$ to the strip
  \begin{equation}
    S_{\e/2} = \{ \log\e-1 < \Re t < 1 \}
  \end{equation}
  such that the images of $(0,1)$ cover $(\log\e,0)$. This is easily
  achieved using affine maps, where the radius of the image is
  taken to be proportional to the distance from the boundary of
  $S_{\e/2}$, and we leave it for the reader to verify that in this
  manner one does obtain a covering using $O(\log|\log\e|)$ maps.

\end{Rem}

\section{Proof of the algebraic lemma}
\label{sec:yg-proof}
We start with a trivial transitivity remark. Assume that $\Phi=\{\phi_\alpha:\cC_\alpha\to X\}$ is a cellular
$r$-parametrization of $X$ and
$\Phi_\alpha=\{\phi_{\alpha,\beta}:\cC_{\alpha,\beta}\to \cC_\alpha\}$ is a
cellular $r$-parametrization of $\cC_\alpha$. Then the collection
$\{\phi_\alpha\circ\phi_{\alpha,\beta}\}$ is ``almost'' a cellular
$r$-parametrization of $X$: by the chain rule
$\norm{\phi_\alpha\circ\phi_{\alpha,\beta}}_r=O_{\ell,r}(1)$ and a
linear subdivision reduces the norms to $1$. We will use this
reduction freely. 

A similar remark holds for $\Phi=\{\phi_\alpha:\cC_\alpha\to X\}$ a
cellular $r$-parametrization of $F:X\to Y$ and
$\Phi_\alpha=\{\phi_{\alpha,\beta}:\cC_{\alpha,\beta}\to \cC_\alpha\}$
a cellular $r$-parametrization of $f_\alpha^*F:\cC_\alpha\to Y$.

We record the following simple lemma.

\begin{Lem}\label{lem:F-vs-Fi}
  Let $n\in\N$ and assume that every definable map $F:X\to I$ with
  $\dim X=n$ admits a cellular $r$-parametrization. Then the same is
  true for every definable map $F:X\to Y$ with $\dim X=n$.
\end{Lem}
\begin{proof}
  Let $\Phi$ be a cellular $r$-parametrization of $F_1$. It will be
  enough to find a cellular $r$-parametrization for $F\circ\phi_1$ for
  each $\phi_1\in\Phi$. In other words we may reduce to the case
  $\norm{F_1}_r=O_r(1)$. We now do the same for $F_2$, noting that
  after the composition we still have $\norm{F_1\circ\phi_2}_r=O_r(1)$
  by the chain rule, and now also $\norm{F_2}_r=O_r(1)$. Repeating
  this for each coordinate we finally get $\norm{F_i}_r=O_{m,r}(1)$
  for every $F_i$ and an additional linear subdivision finishes the
  proof.
\end{proof}

The proof of the Yomdin-Gromov lemma is by induction on
$\ell$. Statement $\rmS_1$ is trivial. We establish $\rmF_1$ as a base
case, and then show
$\rmS_{\le\ell}+\rmF_{\le \ell}\implies \rmS_{\ell+1}$ and
$\rmF_{<\ell}+\rmS_{\le\ell}\implies \rmF_\ell$.

\subsection{Proof of $\rmF_1$}
\label{sec:dim1}

We'll start with a simple lemma due to Gromov about killing
derivatives of univariate functions.

\begin{Lem}\label{lem:yg-univariate-func-step}
  Let $r\ge2$. Suppose that $f:I\to I$ is a definable function
  with $\norm{f}_{r-1}\le 1$. Then $f$ has a cellular
  $r$-parametrization.
\end{Lem}
\begin{proof}
 By  o-minimality we may divide $I$ into finitely many subintervals where $f^{(r)}$ is monotone and 
  continuous. Thus  we assume without loss of generality that $f^{(r)}$  is positive and monotone
  decreasing on $I$. For any $x\in I$
  \begin{equation}
    \frac2x \ge \frac{f^{(r-1)}(x)-f^{(r-1)}(0)}x = f^{(r)}(c_x) \ge f^{(r)}(x)
  \end{equation}
  where $c_x\in(0,x)$ is chosen by the mean-value theorem. Let
  $\tilde f(x)=f(x^2)$. When computing the $\tilde f^{(r)}$ we get a
  bunch of bounded terms plus a term $O_r(x^r f^{(r)}(x^2))$, which is
  bounded by $O_r(x^{r-2})$. Since $r\ge2$ we get
  $\norm{\tilde f}_r=O_r(1)$ and a linear subdivision of $I$ finishes
  the proof.
\end{proof}

We use this to obtain the following.

\begin{Lem}\label{lem:yg-curves}
  Let $X\subset I^2$ be a definable curve. For every $r\in\N$ there
  exists a collection of maps $\{\phi_\alpha:I\to X\}$ such that:
  i) $\cup_\alpha \phi_\alpha(I)=X\setminus\Sigma$ for some finite set
  $\Sigma$; ii) $\norm{\phi_\alpha}_r\le1$ for every $\phi_\alpha$;
  iii) every coordinate of every $\phi_\alpha$ is monotone.
\end{Lem}
\begin{proof}
  By cell decomposition we decompose $X$ into finitely many points,
  intervals $\{x_0\}\times(a,b)$ and graphs of definable functions
  $f:(a,b)\to I$. We denote by $\Sigma$ the set of points, and easily
  parametrize the vertical intervals as required. It remains to
  parametrize the graphs, and we treat each of them separately.

  By o-minimality we may assume that $f$ is either constant (the parametrization is then trivial) or monotone,
  continuously differentiable, and one of
  \begin{align}
    f' \le -1 &&  -1\le f'<0 && 0<f'\le 1 && 1\le f'&
  \end{align}
  holds uniformly. Changing the orientation of $(a,b)$ and exchanging
  the roles of $x$ and $y$ if needed we may assume $0<f'\le1$ in
  $(a,b)$. We are now in position to apply
  Lemma~\ref{lem:yg-univariate-func-step} repeatedly $r-1$ times to
  obtain an $r$-parametrization $\{\tilde{\phi}_\alpha:I\to I\}$ of
  $f$. Setting $\phi_\alpha=(\tilde{\phi}_\alpha,f\circ\tilde{\phi}_\alpha)$ then
  gives the required parametrization of the graph of $f$ (but note the
  $x$ and $y$ coordinates may have been exchanged). Condition (iii)
  follows from the monotonicity of $\tilde{\phi}_\alpha$ and of $f$.
\end{proof}

We are now ready to deduce $\rmF_1$. For the case $q=1$, apply
Lemma~\ref{lem:yg-curves} to the graph of $F$ and let
$\{\phi_\alpha=(\phi^x_\alpha,\phi^y_\alpha)\}$ denote the resulting
collection. Then $\Phi=\{\phi_\alpha^x\}$ (plus the finitely many
points $x$-coordinates of $\Sigma$, covered by zero-dimensional basic
cells) is a cellular $r$-parametrization of $F$. Indeed, it is
cellular by condition (iii), it covers the domain of $F$ since
$\phi_\alpha$ covers the graph by condition (i), and
\begin{equation}
  \norm{F\circ\phi_\alpha^x}_r = \norm{\phi_\alpha^y}_r \le 1
\end{equation}
by condition (ii).

The case of general $q$ now follows by Lemma~\ref{lem:F-vs-Fi}. Note
that we could not have obtained this directly from
Lemma~\ref{lem:yg-univariate-func-step} because the assumption $r\ge2$
is crucial there, and the reduction in Lemma~\ref{lem:yg-curves}
involves changing the order of the variables and is not cellular.

\subsection{The step $\rmS_{\le\ell}+\rmF_{\le\ell}\implies \rmS_{\ell+1}$}

By cell decomposition it is enough to prove the claim for every cell
$C\subset\R^{\ell+1}$. We assume that $C=C_{1..\ell}\odot(a,b)$ (the
case $C=C_{1..\ell}\odot\{a\}$ is similar but easier). By
$\rmF_{\le\ell}$ we may assume that $(a,b)$ already admits a cellular
$r$-parametrization by maps $f_\alpha:\cC_\alpha\to C_{1..\ell}$. Let
$f:\cC\to C_{1..\ell}$ be one of these maps. Then
$\norm{f^*a}_r,\norm{f^*b}_r\le 1$ and setting
\begin{equation}
  \cC':=\cC\times I, \qquad f'(\vx_{1..\ell+1})=(f,\vx_{\ell+1} f^*b+(1-\vx_{\ell+1}) f^*a)
\end{equation}
we have $\norm{f'}_r\le O_\ell(1)$. Taking a linear subdivision of
$\cC'$ finishes the proof.

\subsection{The step $\rmF_{<\ell}+\rmS_{\le\ell}\implies \rmF_\ell$}

\subsubsection{Reduction to $F:I^\ell\to I$}
By Lemma~\ref{lem:F-vs-Fi} it is enough to prove the claim for a
function $F:X\to I$. By $\rmS_\ell$ we may start by taking a cellular
$r$-parametrization of $X$ and reduce without loss of generality to
the case $F:\cC\to I$ for $\cC$ a basic cell of length $\ell$. If
$\cC$ has any $\{0\}$-coordinates we can just ignore them and reduce
to $F_{<\ell}$, so we assume $F:I^\ell\to I$.

\subsubsection{A family version of $\rmF_\ell$}

We will need a ``family version'' of $\rmF_\ell$ as follows.
\begin{description}
\item[$\rmF_\ell$ for families] Let
  $\{F_\lambda:X\to Y\}_{\lambda\in I}$ be a definable family. Then
  there exists (i) a disjoint partition $I=\cup I_j$ into finitely
  many points and intervals; (ii) for every $I_j$ a collection of
  basic cells $\cC_\alpha$ and cellular maps
  $\{f_{\alpha,\lambda}:\cC_\alpha\to X\}_{\lambda\in I_j}$ such that
  (1) $\norm{f_{\alpha,\lambda}}_r\le1$ and
  $\norm{f_{\alpha,\lambda}^*F_\lambda}_r\le1$ for every fixed
  $\lambda\in I_\alpha$; and (2) for every $\lambda\in I_j$ we have
  $X=\cup_\alpha f_{\alpha,\lambda}(\cC_\alpha)$.
\end{description}
It is not difficult to obtain such a family version by adding
parameters to all of the statements in~\secref{sec:yg-proof}. However,
to simplify the presentation we take a shortcut introduced in the
Pila-Wilkie paper: we show in~\secref{sec:automatic-families} that the
family version of $\rmF_{\ell}$ follows from the regular version by
general o-minimality considerations.

\subsubsection{Reduction to $\norm{F(\vx_1,\cdot)}_r\le 1$ for every
  $\vx_1\in I$}
By the family version of $\rmF_{\ell-1}$ we may, thinking $\vx_1$ as a
parameter, find a cellular $r$-parametrization $\Phi=\{\phi^{\vx_1}_\beta\}$
of $F$ with respect to the $\vx_{2..\mu}$ variables (we consider each
interval $I_j$ separately and rescale back to $I$). Fix one
$\phi^{\vx_1}=\phi^{\vx_1}_\beta$ and set
$\hat F=F\circ(\id,\phi^{\vx_1})$. Then $\norm{\hat F(\vx_1,\cdot)}_r\le1$
for every fixed $\vx_1\in I$. By o-minimality $\hat F$ is
$C^r$-smooth outside a positive-codimension set $V\subset I^\ell$.

We first use $\rmS_\ell$ to find a cellular $r$-parametrization
$\{f_{V,\alpha}:\cC_{V,\alpha}\to V\}$. Each $\cC_{V,\alpha}$ must
have dimension strictly smaller than $\ell$, i.e. it has a
$\{0\}$-coordinate, so we can find a cellular $r$-parametrization for
each $f_{V,\alpha}^*F$ using $F_{<\ell}$ as above.

We now use $\rmS_\ell$ to find a cellular $r$-parametrization
$\{f_\alpha:\cC_\alpha\to I^\ell\setminus V\}$. Fixing one such
$\cC,f$ we note that $f^*\hat F$ is $C^r$-smooth on $\cC$, and
crucially we still have
$\norm{f^*\hat F(\vx_1,\cdot)}_r=O_{\ell,r}(1)$ for every fixed
$\vx_1\in I$ because $\norm{f}_r\le 1$ and $f_1$ does not depend on
$\vx_{2..\ell}$. As before we may assume that $\cC=I^\ell$ and use
linear subdivision to get $\norm{f^*\hat F(\vx_1,\cdot)}_r\le 1$.

\subsubsection{Induction over the first unbounded derivative $\valpha$}

We return to our original notation replacing $F$ by $f^*\hat F$. We
may now assume that $F:I^\ell\to I$ is $C^r$-smooth and
$\norm{F(\vx_1,\cdot)}_r\le1$ for every $\vx_1\in I$. Let
$\valpha\in\N^\ell$ be the first index, in degree-lexicographic order,
such that $|\valpha|\le r$ and $\norm{F^{(\valpha)}}>1$. If no such
$\valpha$ exists we are done. We will reparametrize $F$ by cellular
$r$-maps such that the pullback has strictly larger $\valpha$ and then
finish the argument by induction on $\valpha$.

\subsubsection{Reparametrization of the $\vx_1$ variable}
By assumption $\valpha_1>0$. Using Lemma~\ref{lem:bdd-derivative} and
treating the finitely many exceptional $\vx_1$ values by induction on
$\ell$, we may assume without loss of generality that
$F^{(\valpha)}(\vx_1,\cdot)$ is bounded for every
$\vx_1\in I$. Define
\begin{equation}
  S := \{ \vx_{1..\mu} \in I^\ell : |F^{(\valpha)}(\vx_{1..\mu})| \ge \tfrac12
  \sup_{I^{\ell-1}} |F^{(\valpha)}(\vx_1,\cdot)| \}
\end{equation}
Choose a definable curve $\gamma:I\to S$ such that
$\gamma_1(\vx_1)=\vx_1$. Using $\rmF_1$ we find a cellular
$r$-parametrization $\Phi$ of
$(\gamma,F^{(\valpha-1_1)}\circ\gamma)$. Fix $\phi\in\Phi$ and set
$\tilde F:=F\circ(\phi,\id)$.

\subsubsection{Finishing up: a bound on all derivatives up to $\valpha$}
Recall that all derivatives of $(\phi,\id)$ up to order $r$ and all
derivatives $F^{(\vbeta)}$ with $\vbeta<\valpha$ are bounded by
$1$. It follows easily using the chain rule that
$\tilde F^{(\vbeta)}=O_{\ell,r}(1)$ for $\vbeta<\valpha$. Computing
$\tilde F^{(\valpha)}$ we get a bunch of terms that add up to
$O_{\ell,r}(1)$, plus the term
$(\phi')^{\valpha_1} \cdot F^{(\valpha)}\circ(\phi,\id)$. Now
\begin{equation}
  (\phi')^{\valpha_1} \cdot F^{(\valpha)}\circ(\phi,\id) \le
  (\phi')^{\valpha_1} \cdot 2 F^{(\valpha)}\circ\gamma\circ\phi \le
  2\phi'\cdot F^{(\valpha)}\circ\gamma\circ\phi
\end{equation}
since $|\phi'|\le1$ and $\valpha_1\ge1$. To bound the right hand side
we compute
\begin{equation}
  (F^{(\valpha-1_1)}\circ\gamma\circ\phi)' =
  \phi'\cdot\left(F^{(\valpha)}\circ\gamma +\sum_{j=2}^\mu \gamma_j' \cdot F^{(\valpha-1_1+1_j)}\circ\gamma \right)\circ\phi
\end{equation}
and note that the left hand side, $\phi'\cdot\gamma_j'\circ\phi$ and
$F^{(\valpha-1_1+1_j)}$ are $O_{\ell,r}(1)$. Therefore
$\phi'\cdot F^{(\valpha)}\circ\gamma\circ\phi$ is also
$O_{\ell,r}(1)$, and a further subdivision and linear
reparametrization finishes our induction on $\valpha$.

\subsection{Boundedness of derivatives}
\label{sec:bounded-derivatives}

In this section we prove a simple lemma on boundedness of derivatives
that is used in the proof of Theorem~\ref{thm:ygbn}. We let $\mu$ denote
the Lebesgue measure (or just sum of lengths of intervals).

\begin{Lem}\label{lemma:bdd-length}
  Let $\{f_\e(t):I\to I\}$ be a definable family of functions
  depending on a parameter $\e$. Then for every $\e$,
  \begin{equation}\label{eq:big-fdot}
    \mu\big(\{ t\in I : |f'_\e(t)|>M \}\big) < \frac C M
  \end{equation}
  where $C$ is a constant independent of $\e$.
\end{Lem}
\begin{proof}
  The set where $f_\e'(t)>M$ (resp. $f'(t)<-M$) is a union of
  intervals, with their number uniformly bounded by o-minimality, and
  each of length at most $1/M$: otherwise $f_\e$ would leave $I$ along
  such an interval.
\end{proof}

\begin{Lem}\label{lem:bdd-derivative}
  Let $f:I^\ell\to I$ be definable, and suppose that
  $\norm{f'_j}\le1$ for $j=2,\ldots,\ell$. Then the function
  $f'_1(\vx_1,\cdot)$ is bounded for almost every fixed
  $\vx_1\in I$.
\end{Lem}
\begin{proof}
  Assume the contrary. Then the set
  \begin{equation}
    \{ \vx_1\in I : |f'_1(\vx_1,\cdot)| \text{ is unbounded}\}
  \end{equation}
  contains an interval, and we might as well assume after restriction
  and rescaling that it is $I$. For each $M$ we can choose a curve
  $\gamma_M:I\to I^{\ell-1}$ such that
  \begin{equation}
    |f'_1(t,\gamma_M(t))|>M \qquad \forall t\in I,
  \end{equation}
  and we may further assume the dependence on $M$ is definable.
  Applying Lemma~\ref{lemma:bdd-length} to the coordinates of
  $\gamma_M(t)$ and to $f(t,\gamma_M(t))$ we see that outside a set of
  measure $\tilde C/M$ we have $\norm{\gamma_M'(t)}\le M/(3\ell)$ as
  well as
  \begin{gather}
    |f'_{2..\ell}(t,\gamma_M(t))\cdot\gamma_M'(t) + f'_1(t,\gamma_M(t))| = |f(t,\gamma_M(t))'| \le M/3
  \end{gather}
  This is impossible as soon as $M>\tilde C$: the first summand in the left
  hand side is bounded by $M/3$, and the second is at least $M$.
\end{proof} 

\subsection{Automatic uniformity over families}
\label{sec:automatic-families}

In this section we give a model-theoretic proof that the statement
$\rmF_\ell$ for an arbitrary o-minimal structure (and fixed
$\ell,r\in\N$) implies the family version for an arbitrary o-minimal
structure (with the same $\ell,r$). This is the approach employed by
Pila and Wilkie \cite{pila-wilkie}, and we repeat it here with some
more explicit details for non-experts in model-theory who are
nevertheless interested in understanding the mechanics of this general
reduction. However, a reader unfamiliar with the relevant notions from
model theory can alternatively check that the family versions can be
proven in the same manner as the usual statements, essentially
verbatim.

Let $\cM$ be an o-minimal structure and consider a family
$\{F_\lambda:X\to Y\}_{\lambda\in I}$. Let $\cL$ be the language of
$\cM$. Let $\Phi:=\{\phi_\alpha(\vp,\va)\}$ denote the set of all
$\cL$-formulas in two sets of variables and $N\in\N$. For every
$\vphi\in\Phi^N$ we can write the first order formula
$\psi_\vphi(\lambda)$ stating that ``there exists $\vp$ such that that
the formulas $\vphi_1(\vp,\cdot),\ldots,\vphi_N(\vp,\cdot)$ define $N$
cellular maps $f_1,\ldots,f_N:X\to Y$ which form a cellular
$r$-parametrization of $F_\lambda$''. We claim that there are
$\vphi^1,\ldots,\vphi^q$ be such that
$\forall\lambda\in I:\lor_{j=1}^q \psi_{\vphi^j}(\lambda)$ holds in
$\cM$.

Suppose not. Let $c$ denote a new constant and consider the theory
\begin{equation}
  T:=\Th_\cL(\cM)\cup\{c\in I\}\cup\{\lnot\psi_\vphi(c) : N\in\N, \vphi\in\Phi^N \}
\end{equation}
This theory is finitely consistent by our assumption (in fact an
interpretation for $c$ exists in $\cM$). It is therefore consistent by
compactness, and we have an elementary extension $\cM\subset\tilde\cM$
which is again an o-minimal structure. But the axioms of $T$ state
that $F_c^{\tilde\cM}$ has no cellular $r$-parametrization, and this
contradicts $\rmF_\ell$ for $\tilde\cM$.

Now choose $\vphi^1,\ldots,\vphi^q$ as above and set
$I_j:=\{\lambda\in I:\psi_{\vphi^j}(\lambda)\}$. By definable choice
there's a definable map $\lambda\to\vp(\lambda)$ such that for every
$\lambda\in I_j$ the formulas
\begin{equation}
  \vphi^j_1(\vp(\lambda),\vx), \ldots, \vphi^j_{N_j}(\vp(\lambda),\vx)
\end{equation}
define $N_j$ cellular maps which form a cellular $r$-parametrization
of $F_\lambda$. Finally $\cup_j I_j=I$, and refining this into a
partition by points/intervals using o-minimality proves the claim.

\bibliographystyle{plain}
\bibliography{nrefs}

\begin{thebibliography}{10}

\bibitem{CCStructures}
Gal Binyamini and Dmitry Novikov.
\newblock Complex cellular structures.
\newblock {\em Ann. of Math. (2)}, 190(1):145--248, 2019.

\bibitem{burguet:alg-lemma}
David Burguet.
\newblock A proof of {Y}omdin-{G}romov's algebraic lemma.
\newblock {\em Israel J. Math.}, 168:291--316, 2008.

\bibitem{cpw:params}
Raf Cluckers, Jonathan Pila, and Alex Wilkie.
\newblock Uniform parameterization of subanalytic sets and diophantine
  applications, 2016.

\bibitem{gromov:gy}
M.~Gromov.
\newblock Entropy, homology and semialgebraic geometry.
\newblock {\em Ast\'erisque}, (145-146):5, 225--240, 1987.
\newblock S{\'e}minaire Bourbaki, Vol. 1985/86.

\bibitem{kpv:alg-lemma}
Beata Kocel-Cynk, Wies\l~aw Paw\l~ucki, and Anna Valette.
\newblock {$\mathcal{C}^p$}-parametrization in o-minimal structures.
\newblock {\em Canad. Math. Bull.}, 62(1):99--108, 2019.

\bibitem{milnor:dynamics}
John Milnor.
\newblock {\em Dynamics in one complex variable}, volume 160 of {\em Annals of
  Mathematics Studies}.
\newblock Princeton University Press, Princeton, NJ, third edition, 2006.

\bibitem{pila-wilkie}
J.~Pila and A.~J. Wilkie.
\newblock The rational points of a definable set.
\newblock {\em Duke Math. J.}, 133(3):591--616, 2006.

\bibitem{vdd:book}
Lou van~den Dries.
\newblock {\em Tame topology and o-minimal structures}, volume 248 of {\em
  London Mathematical Society Lecture Note Series}.
\newblock Cambridge University Press, Cambridge, 1998.

\bibitem{yomdin:gy}
Y.~Yomdin.
\newblock {$C\sp k$}-resolution of semialgebraic mappings. {A}ddendum to:
  ``{V}olume growth and entropy''.
\newblock {\em Israel J. Math.}, 57(3):301--317, 1987.

\bibitem{yomdin:entropy}
Y.~Yomdin.
\newblock Volume growth and entropy.
\newblock {\em Israel J. Math.}, 57(3):285--300, 1987.

\end{thebibliography}

\end{document}